\documentclass{amsart}
\usepackage{amsmath,amsthm,amssymb,xypic}
\usepackage[all]{xy}
\theoremstyle{plain}

\newtheorem{theorem}{Theorem}[section]

\newtheorem{lemma}[theorem]{Lemma}

\newtheorem{example}[theorem]{Example}

\theoremstyle{definition}

\newtheorem{definition}[theorem]{Definition}

\newtheorem{notation}[theorem]{Notation}

\theoremstyle{remark}

\newtheorem*{claim}{Claim}

\newtheorem{con}[theorem]{Construction}

\newtheorem{remark}[theorem]{Remark}

\DeclareMathOperator{\mult}{mult}

\DeclareMathOperator{\rank}{Rank}

\newcommand{\QED}{\ifhmode\unskip\nobreak\fi\quad {\rm Q.E.D.}} 

\newcommand\Span[1]{\langle{#1}\rangle}

\newcommand\iso{\cong}

\newcommand{\C}{\mathbb{C}}

\newcommand{\F}{\mathbb{F}}

\newcommand{\h}{\mathcal{H}}

\newcommand{\I}{\mathcal{I}}

\renewcommand{\L}{\mathcal{L}}

\renewcommand{\O}{\mathcal{O}}
\renewcommand{\o}{\mathcal{O}}
\renewcommand{\P}{\mathbb{P}}

\newcommand{\p}{\mathbb{P}}

\newcommand{\rat}{\dasharrow}

\begin{document}
\title{Cremona equivalence and log Kodaira dimension}

\author{
Massimiliano Mella}
\address{
Dipartimento di Matematica\ e Informatica\ 
Universit\`a di Ferrara\\
Via Machiavelli 30\\
44121 Ferrara Italia}
\email{mll@unife.it}

\date{September 2025}
\subjclass{Primary 14E07 ; Secondary 14E05, 14J30, 14E30}
\keywords{Cremona transformation, singularities, birational maps,
  log-Kodaira dimension}
\thanks{The author is a member of GNSAGA and has been supported by
  PRIN 2022 project  2022ZRRL4C
  Multilinear Algebraic Geometry.
finanziato dall'Unione Europea - Next Generation EU}
\maketitle
\begin{abstract}
  Two projective varieties are said to be Cremona equivalent if there
  is a Cremona modification sending one onto the other. In the last
  decade, Cremona equivalence has been investigated widely, and we now have a complete theory for non-divisorial reduced schemes. The case
  of irreducible divisors is completely different, and not much is
  known besides the case of plane curves and a few classes of
  surfaces. In particular, for plane curves it is a classical result
  that an irreducible  plane curve is Cremona equivalent to a line if
  and only if its log-Kodaira dimension is negative.  This can be
  interpreted as the log version of Castelnuovo's rationality criterion
  for surfaces. One expects that a similar result for surfaces in
  projective space should not be true, as it is false, the
  generalization in higher dimensions of Castelnuovo's Rationality Theorem.  
In this paper, the first example of such behaviour is provided, exhibiting a rational
surface in the projective space with negative log-Kodaira dimension, which is not Cremona equivalent to a plane. This can be thought of as a sort of log Iskovkikh-Manin, Clemens-Griffith, Artin-Mumford
example. Using this example, it is then possible to show that Cremona
equivalence to a plane is neither open nor closed among log pairs with
negative Kodaira dimension.
\end{abstract}

\section{Introduction}
Let $X,Y\subset \p^N$ be irreducible birational subvarieties. It is
quite natural to ask if there is a birational self-map of the
projective space $\omega:\p^N\dasharrow\p^N$ such that
$\omega(X)=Y$. If this is the case, $X$ is said to be Cremona
equivalent to $Y$.

The notion of Cremona equivalence is old and already at the end
of the $\rm XIX^{th}$ century, both the Italian and the English school of algebraic
geometry approached the problem, with special regard to the plane
curves.

In this context, a lot of attention was devoted to rational curves
Cremona equivalent to a line. It is easy to give examples of rational
curves that are not Cremona equivalent to a line, for instance, a nodal sextic. It is less immediate to
characterize rational curves that are Cremona equivalent to a line. I am not going to explain the long and intricate story of this theorem
and its proofs. Let me simply say that the key idea is that the vanishing of all the adjoints, in
modern language, the negativity of the log Kodaira dimension, is enough to
provide singularities that force the Cremona equivalence of a rational curve to a line.
\begin{theorem}\label{thm:C}
 \cite[vol III pg 188]{CE15} \cite[pg 406]{Co59} A rational curve $C\subset\p^2$ is Cremona equivalent to
  a line if and only if all adjoint linear systems to $C$ vanish,
  i.e. $\overline{\kappa}(\p^2,C)<0$.
\end{theorem}
Where $\overline{\kappa}(\p^2,C)$ is the log Kodaira dimension of the
pair $(\p^2,C)$, see the next section for a precise definition.
This result can be seen as a log version of {Cas\-tel\-nuo\-vo}
rationality criterion, saying that to detect the Cremona equivalence to
a line it is enough to check the vanishing of some prescribed
cohomological groups. Note that \cite{KM82} improved Coolidge result and
proved that it is enough to check the vanishing of the second adjoint,
increasing the similarity  to the Castelnuovo criterion.

My aim is to investigate the higher-dimensional version of Theorem~\ref{thm:C}. First, one should ask for the Cremona equivalence of arbitrary
rational subvarieties of $\p^n$. It is amazing, but as soon as
the codimension is at least 2, any birational equivalence  of reduced schemes
can be obtained via a Cremona modification. This is the content of a
series of papers I dedicated to the subject, \cite{MP09} \cite{CCMRZ16}
\cite{Me22}. For divisors, the situation is much more intricate and
only very special examples of Cremona equivalence are known, \cite{MP09} \cite{Me13}\cite{CC17}\cite{CC18} \cite{Me20} \cite{Me21}.

Like the failure of a higher-dimensional version of Castelnuovo
rationality criterion, proved in a series of seminal papers \cite{IM71} \cite{CG72} \cite{AM72}, the existence of rational surfaces with negative
log Kodaira dimension that are not Cremona equivalent to a plane in
$\P^3$ was largely expected. On the other hand, since no known birational invariant can distinguish those pairs from a plane in
$\p^3$, no examples were known.

The aim of this paper is to provide the first example of such
a behavior. Let $W\subset\p^7$ be the minimal degree embedding of
$\F_0$, that is the embedding given by the linear system $\o_{\F_0}(1,3)$.
\begin{theorem}\label{thm:main}
  Let $S\subset\p^3$ be a general linear projection of $W$. Then
  $\overline{\kappa}(S,\p^3)<0$ and $S$ is not Cremona equivalent to a plane.
\end{theorem}

Let me spend some words on the reason I focused on this special
surface. Any rational surface of degree at most $4$ is Cremona
equivalent to a plane, \cite{Me20}. All rational quintic surfaces I was able to
test are Cremona equivalent to a plane. Unfortunately, the
classification of rational quintic surfaces is incomplete, so I do not have a full statement in this degree. All but $S$, sextic
rational surfaces I studied are Cremona equivalent to a
plane.
  
The surface $S$ can be detected with a 2-ray
game inherited by the Sarkisov program.
My heuristic approach is based on the following pattern. Produce
Sarkisov links out of the singularities of the surface and then
perform a 2-ray game.  In the context of Cremona
equivalence, the linear system that provides the map  is hidden, the starting surface is usually an irreducible component of a very special
divisor in the linear system that realizes the equivalence.
For this reason, this approach alone can only provide positive
answers in very special situations. 
Nonetheless, the surface $S$ is the first case of my personal zoo, where this method
failed to produce a non-terminal 3-fold as output. That is, the 2-ray game led me to a non-terminal variety. I
interpreted this failure as a suggestion to go deeper into the geometry
of the pair.

Once the attention is focused on this example, a brute force analysis of  the intersection theoretic behavior of a possible Cremona equivalence between $S$ and a plane allows us to prove
Theorem~\ref{thm:main}. This, again, reminds me of the first proof of the non-rationality of quartic 3-folds, \cite{IM71}. Unfortunately, at the
moment, a more conceptual approach is not at hand, and this lack
prevents a more general treatment of Cremona equivalence of surfaces with negative Kodaira
dimension.

The example also allows us to prove that being Cremona
equivalent is neither open nor closed, even for families of pairs with negative log Kodaira
dimension, see
Theorems~\ref{th:noopen}~and~\ref{th:noclose} for the precise statements.

The paper is constructed as follows. In Section 2, general results about  Cremona equivalence are
proposed; in particular, Lemma~\ref{lem:blowup} provides a very useful
log resolution of the pair $(\p^3,S)$. The section ends with an
application of the 2-ray game to a different projected surface
Cremona equivalent to a plane.  This construction of Cremona
equivalence, even if not strictly necessary, helps
to understand the geometric reason that makes this special example work.
In Section 3, Theorem~\ref{thm:main} is proved. In the final section, special families of log varieties are studied to prove that  Cremona equivalence to a plane is neither open nor closed
among log pairs of negative log-Kodaira dimension.

\noindent{\sc acknowledgment}\ \   I would like to thank Ciro Ciliberto for his constant support on the subject and
for pushing me to detect this special rational surface. Many thanks
are due to the referee for improving the paper with many suggestions
and for providing references and classical views in various places.

\section{Cremona equivalence: definition and first results}
We work over the complex field.
Let us start introducing the main relation we are going to analyze.
\begin{definition}
  Let $X,Y\subset\p^N$ be two birational subvarieties. We say that $X$ is Cremona equivalent to $Y$ if there is a birational modification
  $\omega:\p^N\dasharrow \p^N$, defined on the generic point of $X$,
  such that $\omega(X)=Y$.
\end{definition}
For a log pair $(\p^n,X)$  we define $\overline{\kappa}(\p^n,X)$ as follows.
\begin{definition}
  \label{def:logK} Let $X\subset\p^n$ be a divisor and let  $\mu:Z\to\p^n$ be a log resolution of the pair, that is, $Z$ is smooth
  and $\mu^*X=X_Z+\Delta$, for some effective $\mu$-exceptional
  divisor $\Delta$, is a normal crossing divisor. We define
  $\overline{\kappa}(\p^n,X)$ to be  the log Kodaira
  dimension of the pair  $(Z,X_Z)$, that is
  $$\overline{\kappa}(\p^n,X):=\kappa(Z,X_Z).$$
\end{definition}
\begin{remark}
  It is well known that Definition~\ref{def:logK} does not
  depend on the log resolution chosen and therefore
  $\overline{\kappa}(\p^n,X)$ is well defined.

  When $\overline{\kappa}(\p^n,X)<0$ all linear systems $m(K_Z+X_Z)$
  are without sections, for $m\geq 1$.
   Note that since $X_Z$ is effective, this is equivalent
  to having
  $$H^0(Z,aK_Z+bX_Z)=0,$$
  for all $a\geq b\geq 1$. This is exactly the vanishing of all
  adjoint linear systems mentioned in Theorem~\ref{thm:C}.
\end{remark}

\begin{definition}
  For a linear system $\L$ on a variety $X$ let $\varphi_\L:X\rat
\L^*$ be the map induced by divisors in $\L$.
\end{definition}

In this paper, we are focused on rational surfaces $S\subset \p^3$ and in
particularly on the generic projection of smooth surfaces to $\p^3$.

\subsection{Cremona equivalence for projected surfaces in $\p^3$}

Let us start by recalling some known results on projected surfaces, see
for instance  \cite{P05}.

Let $f: W\to S\subset\P^3$ be a generic projection of a smooth
surface $W\subset\P^n$. The singular locus of $S$ is a curve $\Gamma$
whose only singularities are  $t$  ordinary triple points with
transverse tangent directions. The surface $S$ has ordinary triple points at the singular
points of $\Gamma$. The  curve $\Gamma$ is a
curve of ordinary double points for $S$ except for a bunch of pinch points. 
Except in the case of generic projections of the Veronese surface,
the curve $\Gamma$ is irreducible. 

The next Lemma is the key to understanding the blow-up of the singular
curve $\Gamma\subset S\subset\p^3$.

\begin{lemma}\label{lem:blowup}
  Let $\Gamma\subset X$ be an irreducible reduced curve in a smooth
  3-fold $X$.  Assume that the only singularities of
  $\Gamma$ are $t$ ordinary triple points. Let $\nu:T\to X$ be the blow up of
  $\Gamma$. Then $T$ is a variety with terminal singularities and the
  only singularities of $T$ are $t$ points of type $1/2(1,-1,1)$ over
  the singularities of $\Gamma$.

  Assume that  $S:=\pi(W)\subset\P^3$ is a general projection of a smooth surface
  $W\subset\P^n$, different from the Veronese surface, and $\Gamma$ its curve of singularities. Let $\nu:T\to\p^3$ be
  the blow up of $\Gamma$ and $S_T$ the strict transform of $S$. Then
  $S_T\cong W$, $\nu_{|S_T}=\pi_{|W}$ is the projection and $S_T$ is on the smooth locus of $T$.
\end{lemma}
\begin{proof} Let $p\in \Gamma$ be a triple point and $F_p$ the
  corresponding fiber in $T$. It is clear that it is enough to prove that $T$ is
  terminal with the required singularities in a neighborhood of $F_p$.

  We can therefore assume, without loss of generality, that the
  curve $\Gamma$ has 
  a unique triple point, say $p$.
 Consider the blow up of the point $p$, say $\nu_p:X_p\to X$,
  with exceptional divisor $E_p$ and
  then the blow up of the curve $\Gamma$, $\nu_{\Gamma}:X_{\Gamma}\to X_p$,
  with exceptional divisor $E_\Gamma$.

  By hypothesis $\nu_\Gamma$ restrict to $E_p$ as the blow up of 3
  not aligned points. Let $r_i$ be the lines joining the 3
  points. Then we have
  $$K_{X_\Gamma}\cdot r_i=0\ \ {\rm and}\ \
  N_{r_i/X_\Gamma}\cong\o(-1)\oplus\o(-1).$$
 Therefore, we may flop these
 lines with a map  $\psi:X_\Gamma\rat Y$.

 By construction, $\psi_{|E_p}$ is
  the contraction of the three lines. In particular,  the composition
  $(\psi\circ\nu_\Gamma)_{|E_p}$ is a standard Cremona transformation
  for the plane $E_p$. Let $E\subset Y$ be the strict transform of $E_p$ then
  $$E_{|E}\sim\O(-2). $$
 This shows that we may blow down $E$ to a singular point of type
 $1/2(1,-1,1)$. Let $\phi:Y\to Y_\Gamma$ be the blow down of $E$.
Note that the construction is relative over $X$. Therefore, there is a
canonical morphism $g:Y_\Gamma\to X$. The
 irreducibility of $\Gamma$ yields
 $$\rank
 Pic(Y_\Gamma/X)=1.$$
 Therefore, $Y_\Gamma$ is the terminal elementary
 extraction of $X$ and $g=\nu:Y_\Gamma\to X$ is the unique blow-up of the curve $\Gamma$. 
 This concludes the first part of the proof.

 Let us summarize all the maps involved in the following diagram

   \[
 \xymatrix{
   &X_{\Gamma}\ar[dl]_{\nu_\Gamma}\ar@{.>}[r]^\psi&Y\ar[d]^\phi  \\
 X_p\ar[d]_{\nu_p}&              &Y_\Gamma\ar[dll]_\nu\\
 X& &}
\]

Let $\pi:W\to S$ be the linear projection. By hypothesis, $W$ is smooth
and $\pi$ is a finite birational map. In particular, $W$ is the
normalization of $S$.
To prove the latter statement in the Lemma, let us follow the birational
modifications on the strict transforms of $S$ along the blow-up
diagram.

The surface $S$ has  triple ordinary points on the singularities of
the curve $\Gamma$ and double points on the smooth points of
$\Gamma$. The latter are ordinary double points except for finitely many
cusps.

Then $S_p\subset
X_p$ is singular along $\Gamma$ and 
$$S_{p|E_p}=r_1+r_2+r_3,$$
being a cubic singular in the three points of intersection $\{r_i\cap
r_j\}_{i\neq j}$.

The
surface 
$S_\Gamma\subset X_\Gamma$ is smooth along $E_p$ and $r_i$ is a local
complete intersection of $E_\Gamma$ and $S_\Gamma$. In particular, the self-intersection of $r_i$ in $S_\Gamma$ is $(-1)$. This shows that
$\psi_{|S_\Gamma}:S_\Gamma\to S_Y$
is the blow down of three $(-1)$-curves and $S_Y\cap
E_p=\emptyset$. Let $S_{Y_\Gamma}=\phi(S_Y)$ be the image of
$S_Y$. Then  the surface $S_{Y_\Gamma}$ is smooth, and it is on
the smooth locus of $Y_\Gamma$.  This shows that 
$\nu_{|S_{Y_\Gamma}}:S_{Y_\Gamma}\to S$ is a finite birational morphism.  In
particular $\nu_{|S_{Y_\Gamma}}$ is the normalization of $S$ and we
conclude 
$$W\cong S$$
and $\nu_{|W}=\pi$.
\end{proof}

\begin{remark}\label{rem:rk2}
  The main feature of Lemma~\ref{lem:blowup} we are going to use is the fact that
  the log variety $(T,S_T)$ has $\rank Pic(T)=2$ and $S_T$ is smooth. This
  drastically simplifies the computations needed to study Cremona Equivalence.
\end{remark}

It is time to shed light on the surface we are looking for.
\begin{notation}\label{notation:W} Let $W\subset\P^7$ be the Segre-Veronese embedding of
  $\F_0$ with the linear system $\o_{\F_0}(1,3)$ and $\pi:W\to S\subset\p^3$ a
  general linear projection of $W$. Then $S$ is a surface of degree $6$
  singular along an irreducible curve  $\Gamma$.

  \begin{remark}
    Note that the class of rational ruled surfaces in $\P^3$ has been
    extensively studied  by Edge, \cite[Chapter IV]{Ed31}, and many of the
    results we are going to use were known to Zeuthen. In particular, the
    one we are considering is a ``general'' sextic ruled
    surface. Besides the results in \cite{Ed31}, one may find much
    information and all the properties we need in Dolgachev's book \cite{Do12}.
  \end{remark}
  The degree of $\Gamma$ can be computed via the sectional genus
 of $W$. Indeed, a general plane section of $S$ is a plane sextic curve of geometric
genus $g(W)=0$ with $\deg\Gamma$ ordinary double points. That is
\begin{equation}
  \label{eq:gradoGammagenus}
  \deg\Gamma={5\choose 2}=10.
\end{equation}
For what follows, it is not crucial to know the number of triple points
and cusps, but they are respectively $4$ and $8$, see either
\cite{P05} or \cite[Theorem 10.4.9]{Do12}.
Let $\Gamma_W=\pi^{-1}(\Gamma)$ be the double point locus of the
projection $\pi$. Then $\Gamma_W\subset\F_0$ is a curve of degree $20$.
Set $\nu:T\to\P^3$ the blow up of $\Gamma$  with exceptional divisor
  $E_\Gamma$. Then by Lemma~\ref{lem:blowup} $S_T$ is a smooth quadric.
\end{notation}
Firstly, we determine the class of the divisor  $\Gamma_W$, see also
the double point class formula in \cite[Equation 10.52]{Do12}. 
\begin{lemma}
  \label{lem:gammaintersection_lines} Let $l\subset S$ be a general
  line then
  $$\Gamma\cap l= 4$$
   and $\Gamma_W\sim\o_{\F_0}(4,8)$.
\end{lemma}
\begin{proof} Let $r:=\pi^{-1}(l)\subset W$ be the preimage and let
  $H\subset\P^7$  be a general hyperplane containing $r$.
Then $r^2=0$ and $H_{|W}=r+R$ for some residual curve
$R$. In particular, we have
$$1=r\cdot H=r^2+r\cdot R.$$
This shows that
$$r\cdot R=1.$$
Let $P\subset\P^3$ be a general plane containing $l$. Then
$P_{|S}=l+D$, for $D$ a plane curve of degree $5$.  That is  $D\cdot
l=5$.

Combining the two equations, we get
$$r\cdot \Gamma_W=D\cdot l-r\cdot R= 4.$$ Since the projection is general
$$r\cdot\Gamma_W=l\cap \Gamma=4.$$
In particular $\Gamma_W\sim \o(4,a)$ and since $\deg\Gamma_W=20$ we
conclude
$$20=\o(4,a)\cdot\o(1,3)=12+a.$$
That is 
$\Gamma_W\sim\o(4,8)$.
\end{proof}

Next, we compute the log Kodaira dimension of $(\p^3,S)$
\begin{lemma}\label{lem:Kneg} In the above notation we have
  $\overline{\kappa}(\p^3,S)<0$.
\end{lemma}
\begin{proof}Let $\nu:T\to\p^3$ be the blow up of $\Gamma$, with
  exceptional divisor $E_\Gamma$ and $S_T$
  the strict transform of $S$.
  By Lemma~\ref{lem:blowup} $\overline{\kappa}(\p^3,S)=\kappa(T,S_T)$.
  We have
  $$K_T+S_T=\nu^*\o_{\p^3}(2)-E_\Gamma,$$ therefore to conclude, we have to prove that there is no surface
in $\p^3$ of degree $d$ having multiplicity at least $\frac{d}2$ along
$\Gamma$. Assume that such a surface exists and call it $D$. Since the
multiplicity of $S$ along $\Gamma$ is less than $\frac62$, we may
assume that $D\not\supset S$.
By construction $D\cdot S$ contains $\Gamma$ of multiplicity at least
$\deg D$, hence we obtain the contradiction
$$\deg(D\cdot S)=6\deg D\geq \deg D\deg \Gamma=10\deg D.$$
\end{proof}

Let us now apply Lemma~\ref{lem:blowup} and
Lemma~\ref{lem:gammaintersection_lines} to study the 2-ray game
originated by the blow up of $\Gamma$.

Let $\nu:T\to\p^3$ be the
blow up of the singular curve $\Gamma\subset S$, with exceptional
divisor $E_\Gamma$.
Then $\nu^*S=S_T+2E_\Gamma$, $Pic(T)$ has rank 2 and the strict transform $S_T$ is a smooth
quadric, by Lemma~\ref{lem:blowup}. Let $f_1$ and $f_2$ be the two
rulings, then, by Lemma~\ref{lem:gammaintersection_lines},
$$S_T\cdot f_1=18-16=2,\ \ S_T\cdot f_2=6-8=-2.$$
This shows that the cone of curves $NE(T)$ is spanned by:  fibers of the
blow up and the class of a curve  contained in $S$. On the other hand, $S$
is a smooth quadric and $S\cdot f_2<0$ therefore $[f_2]$ spans the second
ray. The $2$-ray game will force us to contract $f_2$, with a morphism
$\eta:T\to X$. Since $S_T\cdot f_2=-2$ we have $K_T\cdot f_2=0$.
Hence $X$ is a Fano 3-fold with a curve of canonical singularities.
As already stated, this is not enough to conclude our theorem, see for
instance Example~\ref{ex:dPfib}. Indeed,
there is no guarantee that the blow-up of $\Gamma$ is the first step
of a Sarkisov factorization  based on the linear system, say $\h$, that
produces the Cremona equivalence. This is  due to the fact that we could have
$S+R\in\h$ for some effective non trivial divisor $R$ and the
canonical threshold of the pair $(\p^3,\h)$ could be smaller than
$\frac12$.
On the other hand, it is an indication to study
deeper the geometry of the log pair $(\p^3,S)$.

Let me add a different  example of the $2$-ray game, where, even if the
Sarkisov factorization does not work; the surface is
Cremona equivalent to a plane.

\begin{example}\label{ex:dPfib}
 Let $W\subset\p^6$ be a del Pezzo surface of degree $6$ and
 $\pi:W\to S\subset\p^3$ a general projection. This time $W$ has
 sectional genus $1$ and, as before or via the double point formula
 \cite[Equation 10.52]{Do12}, we prove that   the singular
 curve $\Gamma$  has degree $9$. Let $\nu:T\to\p^3$ be the blow up of
 $\Gamma$, with exceptional divisor $E$ and strict transform
 $S_T$. Then $S_T$ is a del Pezzo surface of degree $6$ and, with a computation
 similar to that in Lemma~\ref{lem:gammaintersection_lines}, we see that  $S_T\cdot m_i=0$ for any $(-1)$-curve $m_i$ in $S_T$.
Note that lines, i.e. $(-1)$-curves, generate $Pic(S_T)$; therefore, $S_T$ is a fiber of a pencil of del Pezzo surfaces of degree 6. Incidentally, note that $36=6^2=9\cdot 4$. The del
   Pezzo fibration is the
   second ray of the two-ray game. Therefore, also in this case, we cannot
   continue the Sarkisov factorization. On the other hand, this time the
   fibration can be used to produce the required Cremona equivalence.

   Let $\pi:T\to\P^1$ be the $dP_6$ fibration and $S_\xi$ the generic
 fiber over $k=\C(t)$. It is well known, \cite{Ma72} see also
 \cite[Chap. IV Theorem 6.8]{Ko96}, that $S_\xi$ is rational over $k$,
 since the Brauer group of a $C_1$-field is trivial. Therefore the pair $(\P^1\times\P^2,\P^2)$ is a birational model of
 $(\P^3,S)$. This is enough to conclude that $(\P^3,S)$ is Cremona equivalent to a
 plane.

In general, I am not trying to  factor the Cremona equivalence in Sarkisov links. I only use this
 technique to improve the  log pairs and understand their geometry better.
\end{example}

The last ingredient we need is the following, probably well-known Lemma, of
independent interest.

\begin{lemma}\label{lem:extending_Cremona}
  Let $\omega:\p^2\dasharrow \p^2$ be a birational map and
  $H\subset\p^3$ is a plane. Then there is a birational map
  $\Omega:\p^3\dasharrow \p^3$ such  that
  $$\Omega_{|H}=\omega.$$
In particular, $\Omega(H)$ is a plane embedded linearly in $\p^3$.
\end{lemma}
\begin{remark}
  Note that Lemma~\ref{lem:extending_Cremona} works in arbitrary
  dimension for linear
  spaces. I thank J\'er\'emy Blanc for the nice remark.
\end{remark}

\begin{proof}
  By Noether--Castelnuovo Theorem, $\omega$ can be factored into a
  sequence of  linear
  automorphisms and standard Cremona transformations. Hence, to
  conclude, it is enough to prove that those are extendible to $\p^3$. For linear
  automorphisms, it is clear. For the latter note that the quadro
  quadric map of $\p^3$, associated to the linear system of quadrics
  containing a point, say $p$,  and a conic, $C$,  restricts to the standard Cremona
  modification on a general plane through the base point $p$.
  To conclude, observe that linear automorphisms preserve linear spaces
  and the quadro-quadric map, centered in a point $p$,  sends planes
  through $p$ to planes.
\end{proof}
We are ready to prove the Theorem stated in the introduction.

\section{The surface $S$ is not Cremona equivalent to a plane}
We prove the statement by contradiction.
Assume that there is a birational map $\chi:\p^3\dasharrow
\p^3$ such that $\chi(S)=H$ is a plane.

Set $\chi=\varphi_{\h}$ and $\chi^{-1}=\varphi_{\h^{\prime}}$, for
linear systems $\h\subset|\o_{\p^3}(h)|$ and
$\h'\subset|\o_{\p^3}(h')|$.
Consider a resolution of the map $\chi$

   \[
 \xymatrix{
   &Z\ar[dl]_p\ar[dr]^q&  \\
 \p^3\ar@{.>}[rr]^\chi&              &\p^3}
\]
and let $S_Z$ be the strict transform of $S$, that coincides with
strict transform of $H$, in $Z$. I'm assuming that  both $Z$ and
$S_Z$ are smooth. This
induces the following  restricted diagram 
  \[
 \xymatrix{
   &S_Z\ar[dl]_{p_{|Z}}\ar[dr]^{q_{|Z}}&  \\
 S\ar@{.>}[rr]^{\chi_{|S}}&              &H}
\]
We aim to study the restricted diagram next.

Recall that the surface $S$ is the projection of $:W\subset\P^7$, where
$W$ is the embedding of $\F_0$ via the linear system
$|\o_{\F_0}(1,3)|$.
Let $\pi:W\to S$ be the projection.
For our purposes, it is also useful to describe $W$ in a different way.

Let  $\L=|\I_{q_1^3\cup q_2}(4)|\subset|\o_{\p^2}(4)|$ be the linear
system of quartics with a 3-ple point in $q_1$ and passing through
$q_2$, then
$$W=\varphi_\L(\p^2).$$
Note that $\varphi_\L^{-1}:W\dasharrow\p^2$ is induced by the linear
system $|\I_{p_1}(1,1)|\subset|\o_{\F_0}(1,1)|$ of rational quartics,
in the embedding of $W\subset\p^7$,  with a simple
base point $p_1$. 
 Then there is a base-point-free linear system of quartics $\Lambda^\prime\subset\L$
 such that
 $$\psi:=\varphi_{\Lambda^\prime}:\p^2\dasharrow S.$$
 
 Set
 $$\omega:=\psi^{-1}\circ(\chi^{-1})_{|H}:\p^2\dasharrow\p^2$$
 and
 $$\Omega:\p^3\dasharrow\p^3$$
 its extensions to $\p^3$, as in
 Lemma~\ref{lem:extending_Cremona}.
 Then by construction $\Omega\circ\chi:\p^3\dasharrow\p^3$ is well
 defined on the generic point of $S$ and 
 $$(\Omega\circ\chi)_{|S}=\psi^{-1}. $$
 
Then, up to replacement $\chi$ with $\Omega\circ\chi$, we may  assume
that $(\chi^{-1})_{|H}=\psi$.
\begin{remark}
  Here we are reviving Cremona's classifying
  method for Cremona transformations of $\P^3$. That is, reduce the
  study of birational modifications of $\P^3$ to the one of plane
  transformations, \cite{Cr871}.
\end{remark}
 In other words, we
 may assume that
 \begin{equation*}
   \label{eq:defHp}
   \h^\prime_{|H}=\Lambda^\prime+F^{\prime},
 \end{equation*}
 for some fixed
 divisor $F^\prime$ and
 \begin{equation*}
   \label{eq:defH}
\h_{|S}=\Lambda+F,
 \end{equation*}
  for some fixed divisor
  $F$ and $\pi^{-1}_*\Lambda\in|\o_{\F_0}(1,1)|$.
  
 Let me stress some interesting consequences :
 \begin{itemize}
\item[-] $S_Z$ is the blow up of $\p^2$ in the points, $q_1, q_2$,
  equivalently the blow up of $S$ along $\Gamma$ and subsequently in a point $p_1$,
\item[-] $p_{|Z}$ is the blowing down of the line $M$ spanned by $q_1$
   and $q_2$ to the point $p_1$, followed by a finite morphism
   \item[-] $q_{|Z}$ is the blowing down of the two rulings, say $F_1$
     and $F_2$, of $S$
     passing through $p_1$,
     \item[-] the rulings of $S$ are mapped to the pencils of lines through
       $q_1$ and $q_2$,
       \item[-] $Pic(S_Z)=\Span{F_1,F_2, M}$, and all the three
         generators are $(-1)$-curves with $F_1\cdot F_2=0$ and $F_i\cdot M=1$.
       \end{itemize}
 A divisor $aF_1+bF_2+cM\in Pic(S_Z)$ will be denoted by $(a,b,c)$. In this notation, the strict transform of the ruling of lines in
 $S$ is $(0,1,1)$ while the strict transform of the ruling of twisted
 cubics corresponds to $(1,0,1)$.  A divisor $D=(a,b,c)$ is effective
 if and only if $a\geq 0$, $b\geq 0$ and $c\geq 0$ and
 \begin{itemize}
 \item[i)]  it is the pull
   back of divisor in $S_T$ if and only if $a+b=c$
   \item[ii)] it is the pull back of a divisor in $H$ if and only if $a=b=c$.
 \end{itemize}
 Let us set some further notation:
 \begin{itemize}
 \item[-] $\nu:T\to \p^3$ the blow up of $\Gamma$, with exceptional
   divisor $E_\Gamma$ and strict transform $S_T$; recall that by
   construction $S_T$ is
   a smooth quadric with $\nu^*\o_{S}(1)\sim\o_{\F_0}(1,3)$ and, by
   Lemma~\ref{lem:gammaintersection_lines} and Lemma~\ref{lem:blowup}, $E_{\Gamma|S_T}\sim(4,8)$,
 \item[-]  $\nu_1:T_1\to T$ the blow up of the fixed component $F$;
the surface $S$ is smooth, then  $F\subset S$ is a Cartier divisor and
   $S_{T_1}\cong S_T\cong {\mathbb F}_0$, by the universal property of blow
   up,\cite[Corollary 7.15]{Ha77}.
   \item[-] $\mu:X\to\p^3$ the blow up the fixed component $F'$; the
   surface $H$ is smooth,  then  $F'\subset H$ is a Cartier divisor and
   $H_{X}\cong H\cong \p^2$, by the universal property of blow
   up,\cite[Corollary 7.15]{Ha77}.
 \end{itemize}
After these blow ups we have  ${\rm Bs}\h_{T_1}\cap S_{T_1}\subset p_1$
and $ {\rm Bs}\h^\prime_X\cap H_X\subset\{q_1,q_2\}$.
Then we have
$$p=\nu_2\circ\nu_1\circ\nu$$
and
$$q=\mu_1\circ\mu,$$
for some birational morphisms $\nu_2:Z\to T_1$ and $\mu_1:Z\to X$ such
that  the restricted morphisms $\nu_{2|S_Z}$ is the blow up of the
point $p_1$ and $\mu_{1|S_Z}$ is the blow up of the points $q_1$,
$q_2$.

By our construction and
Lemma~\ref{lem:gammaintersection_lines},  we have
\begin{equation}
  \label{eq:relations_in_pic}
  \begin{array}{ll}
    ((\nu_2\circ\nu_1)^*E_\Gamma)_{|S_Z}\sim(4,8,12),&\\
    p^*\o_{\p^3}(1)_{|S_Z}\sim (1,3,4),&\\
     q^*\o_{\p^3}(1)_{|S_Z}\sim (1,1,1)
  \end{array}.
\end{equation}
And we may write
\begin{eqnarray}
  \label{eq:p*S}
  p^*S_{|S_Z}=S_{Z|S_Z}+(8,16,24)+ E_{S}+(0,0,a+1)\sim (6,18,24)\\
  \label{eq:q*H}
  q^*H_{|S_Z}=S_{Z|S_Z}+E_{H}+(b_1+1,b_2+1,0)\sim (1,1,1)
\end{eqnarray}
%
where:
\begin{itemize}
\item[-] $E_S$ is the total transform of curves in $S_T$, blown up along the
  map $\nu_1$, 
  \item[-] $E_H$ is the total transform of curves in $H$, blown up along
    the map $\mu$,
    \item[-] $a$ is a non negative integer related to the map $\nu_2$,
      \item[-] $b_1$ and $b_2$ are non negative integers related to  the map $\mu_1$.
\end{itemize}

In particular $E_{S}$ and $E_{H}$ are the pull-back of
curves in $S_T$ and $H$ respectively, then by items i) and ii),  for non-negative
integers $s_1,s_2$ and $e$,   we have
\begin{equation}
  \label{eq:pullback}
  E_S\sim(s_1,s_2,s_1+s_2),\ \ E_{H}\sim(e,e,e).
\end{equation}

From Equations~(\ref{eq:p*S})~(\ref{eq:q*H}) we get 
\begin{equation*}
  E_{H}\sim E_{S}+(2-b_1,-2-b_2,2+a),
\end{equation*}
and finally, plugging in  Equation~(\ref{eq:pullback})
\begin{equation}
  \label{eq:E-E}
  (e,e,e) \sim (s_1+2-b_1,s_2-2-b_2,s_1+s_2+2+a).
\end{equation}
Then we have
\begin{equation}
  \label{eq:LSf}
  \left\{
      \begin{array}{rl}
        e&=s_1+2-b_1\\
        e&=s_2-2-b_2\\
        e&=s_1+s_2+2+a
      \end{array}\right.
\end{equation}
therefore
\begin{equation}
  \label{eq:LSf2}
  \left\{
      \begin{array}{rl}
          e+b_1&=s_1+2\\
        e&=s_2-2-b_2\\
        0&=b_1+s_2+a
      \end{array}\right. .
  \end{equation}
  Since all integers are non-negative, the third equation yields
  $a=b_1=s_2=0$ and from the second equation we derive the impossible $e=-2-b_2$.
This contradiction shows that the map $\chi$ cannot exist and proves
that $S$ is not Cremona equivalent to a plane.

\begin{remark}
  Let me stress that, in the notation of \cite{MP12}, thanks to
  \cite[Corollary 1.7]{Me21}, this shows that
  any good model $(X,S_X)$ birational to $(\p^3,S)$ is such that
  $\rho(X,S_X)=0$. In other words, for any 
  Mori fiber space $(X,S_X)$, with $S_X$ smooth, birational to $(\p^3,S)$, the surface $S_X$
  is never transverse to the MfS fibration. This
  answers to a question in \cite[Remark 4.8]{MP12}.
\end{remark}

\section{Cremona equivalence to a plane is neither open nor closed}
It is easy and not surprising that being Cremona equivalent to a
plane, like rationality for 3-fold hypersurfaces, is neither closed nor open among log pairs, without further hypothesis.
For instance, one can consider a family of smooth cubic surfaces degenerating to a
smooth cubic cone or a family of smooth quartic surfaces degenerating to a rational quartic.

In this section, we use the result in Theorem~\ref{thm:main} to prove
that restricting to  pairs with a negative log Kodaira dimension is not
enough to gain neither openness nor closedness.

\begin{theorem}\label{th:noopen}  There exist families of
  projective surfaces in $\p^3$, $\phi: X\to B$ over connected varieties
  $B$, such that for every $b\in B$ the fiber $X_b = \phi^{-1}(b)$
  satisfies  $\overline{\kappa}(\p^3,X_b)<0$ and for every $b\neq 0$
  $X_b$ is 
  not Cremona equivalent to a plane, while $X_0$ is  Cremona equivalent to a plane.   
\end{theorem}
\begin{proof}
  Let $W\subset\p^7$ be the Segre-Veronese embedding of
  $\p^1\times\p^1$ we considered in Notation~\ref{notation:W}. Fix $5$
  general points $\{x_0,\ldots,x_4\}\subset W$ and let
  $\Lambda=\Span{x_0,\ldots,x_4}$.
  Set $\Pi_0\subset\Lambda$ a general linear space of dimension $3$ and
  $\Pi_1\subset\p^7$ is a general linear space of dimension
  3. Let $B\subset{\mathbf G}(3,7)$ be a rational curve, parameterized
  by $t$, through
  $[\Pi_0]$ and $[\Pi_1]$ in the Grassmannian variety. Let $X_t$ be the
  projection of $W$ from $\Pi_t$, the linear space associated with the
  parameter $t$. Then,
without loss of generality, we may assume that,  for
$t\neq 0$,  $X_t$ is a general projection.
Then,  by Theorem~\ref{thm:main}  for $t\neq 0$, 
$X_t$  is not Cremona equivalent to a
  plane. On the other hand, by construction
  $X_0$ is a sextic with a 5-tuple point, and it is therefore Cremona
  equivalent to a plane, like any monoid, see for instance \cite{Me22}.
\end{proof}

To study closedness, we start by studying the Cremona equivalence of a different
quartic surface, known as the Bordiga surface.

  \begin{lemma}\label{lem:CEdeg6} Fix $10$ general points
    $\{x_1,\ldots, x_{10}\}\subset\p^2$.
Consider the
linear system
$\L:=|\I_{x_1\cup\ldots\cup x_{10}}(4)|$ and let 
$X:=\varphi_\L(\p^2)\subset\p^4$ be the image in $\p^4$.
Let $S'$ be a general linear projection of  $X$ to $\p^3$, then
$\overline{\kappa}(\p^3,S')<0$ and $S'$ is Cremona equivalent to a plane.
  \end{lemma}
  \begin{proof}
Let $\nu':T'\to\p^3$ be the blow up of $\Gamma'$, the singular curve
of $S'$.  Then, by Lemma~\ref{lem:blowup}, the strict transform $S^\prime_{T'}\iso X$ is the blow up of
 $\P^2$ in $Z:=\{p_1,\ldots, p_{10}\}$. The sectional genus of $X$ is $3$
 and therefore  we have $\deg\Gamma'=7$, recall the computation in Notation~\ref{notation:W}.

 \begin{claim}
   $\overline{\kappa}(\p^3,S')<0$
   \end{claim}
   \begin{proof}[Proof of the Claim] As in the proof of
     Lemma~\ref{lem:Kneg}, the claim is equivalent to proving that there are no effective
irreducible divisors
$D\subset\p^3$  such that
$$\mult_\Gamma D\geq \frac{\deg D}2. $$
Assume that such a divisor exists, then $D\cdot S$ has to contain the
curve $\Gamma'$ of multiplicity $\geq \deg D$, Hence we derive the contradiction
$$6\deg D=\deg(D\cdot S)\geq \deg\Gamma'\deg D=7\deg D.$$
   \end{proof}

 Let
 $m_i\subset X$ be the exceptional divisors of the blow-up. Then $m_i\subset
 X$ is a line and arguing as in
 Lemma~\ref{lem:gammaintersection_lines} we have, with $\nu$ the blow
 up of the curve $\Gamma$,
 \begin{equation}
   \label{eq:d6g3}
      \sharp(\Gamma\cap \nu(m_i))=3
 \end{equation}
This gives 
\begin{equation}
  \label{eq:d6g34}
  S^\prime_{T'}\cdot m_i= 6-6=0,\ \ K_{T'}\cdot m_i=-4+3=-1.
\end{equation}

The Picard group of $X$ is generated by the exceptional divisors
$\{m_i\}$ and the strict transform of a  line through two points in $Z$,
say $c\subset X$. Then $c$ is a conic and $c^2=-1$. Let $H\subset\p^4$ be a
hyperplane containing $c$ and $R$ the residual curve
$$H\cap X=c+R.$$
Then we have
\begin{equation}
  \label{eq:d6g32}
  2=(c+R)\cdot c=-1+R\cdot c.
\end{equation}

Looking at this configuration after the projection $\pi$ to $\p^3$
gives:
\begin{itemize}
\item[-] $\pi(c)$ is a conic,
\item[-] $\pi(R)$ is a plane curve of degree 4.
\end{itemize}
This yields
$$\sharp(\pi(R)\cap \pi(c))=8$$
and together with Equation~(\ref{eq:d6g32})
$$\sharp(\Gamma\cap\pi(c))=5.$$
Hence we conclude
\begin{equation}
  \label{eq:d6g33}
  S_{T'}\cdot c=12-10=2,\ \ K_{T'}\cdot c=-8+5=-3.
\end{equation}
Let $C\subset S_T$ be an irreducible effective curve, then
$$C\equiv\alpha
c+\sum_1^{10}\beta_i m_i,$$
and either $\alpha=0$,  and $\sum\beta_i=1$ or  $\alpha>0$. In particular
if $C\neq m_i$
\begin{equation}
  \label{eq:dotpositive6}
  S_{T'}\cdot C>0.
\end{equation}
Equations~(\ref{eq:dotpositive6}) and (\ref{eq:d6g34}) show that $S_{T'}$
is nef.
\begin{claim}
  $T'$ is a Fano 3-fold.
\end{claim}
\begin{proof}[Proof of the Claim] The cone of effective curves has
  dimension 2, $S_{T'}$ is nef and $S_{T'}\cdot m_i=0$.
  Therefore, the cone is closed and
  spanned by the curves contracted by $\nu'$ and the curves
  numerically proportional to $m_i$.
The map $\nu'$ is a terminal blow up, hence $-K_{T'}$ is relatively
ample and by Equation~(\ref{eq:d6g34}) we have
$-K_{T'}\cdot m_i=1.$ Then $T'$ is Fano.
\end{proof}

Let $\mu:T'\to \tilde{T}$ be the contraction of the second extremal ray
$[m_i]$. We have
$S_{T'}\cdot m_i=0$ and by Equation~(\ref{eq:dotpositive6}) we
conclude that  $\mu$ is a
birational contraction. In particular, there is an effective divisor $D\subset T'$
with $D\cdot m_i<0$. Then $S_{T'}$ is in the interior of the
effective cone and the effective threshold
$$\rho(T',S_{T'}):=\inf\{t\in {\mathbb Q}|tK_{T'}+S_{T'}\mbox{ is effective}\}$$
is strictly positive.
In the notation of \cite{Me20}, by Lemma~\ref{lem:blowup}, $(T',S_{T'}
)$ is a good model of the pair $(\p^3,S')$. Therefore, by
\cite[Corollary 1.7]{Me20}  $S$ is Cremona equivalent to a plane.
\end{proof}

Next, we recall  a degeneration argument to study point
collisions. For this, we use notations and results in \cite{GM19}.

\begin{con}[Specialization with $6$ collapsing simple points]\label{not:collapse}
  Set $V=\p^2\times\Delta$,  for $\Delta\ni 0$ a complex disk, and  
  \begin{itemize}
  \item[-]  $\pi:V\to\Delta$, $\tau:V\to\p^2$  the canonical
    projections,
  \item[-] $|\o_V(d)|:=|\tau^*(\o_{\p^2}(d)|$.
    \item[-] $\L_t:=\L_{|\pi^{-1}(t)}$ the restriction to the fiber of a
      linear system $\L$ on $V$.
  \end{itemize}

Fix $6$ general sections  $\{\sigma_1,\ldots,\sigma_{6}\}$ such that
$\sigma_i(0)=p_1$, and set $Z:=\cup_i\sigma_i(\Delta)$.
Let $Y\to V$ be the blow up of $V$ at the point $p_1$, with exceptional divisor $E$. Then we have natural morphisms
$\tau_Y:Y\to\p^2$, a degeneration $\pi_{Y}:Y\to\Delta$, and sections
$\sigma_{Y,i}:\Delta\to Y$. 
The special fiber $W_0$ is given by $E\cup V_0$, 
where $V_0$ is $\p^2$ blown up in one point and $E\cong\p^2$.
Keep in mind that since the sections $\sigma_i$'s are general $\{\sigma_{Y,i}(0)\}$ are general points of $E$. 

We are interested in the  flat  limit scheme $Z_0$. By  \cite[Lemma
20]{GM19} $Z_0$ is a scheme of length $6$ and multiplicity $3$. Then
$Z_0$ is the triple point in $p_1$.
\end{con}

We are ready to prove that Cremona equivalence is not a closed
condition among pairs with negative log Kodaira dimension.

\begin{theorem}\label{th:noclose}  There exist a family of
  projective surfaces in $\p^3$, $\phi: X\to B$ over a connected variety
  $B$, such that for every $b\in B$ the fiber $X_b = \phi^{-1}(b)$
  satisfies  $\overline{\kappa}(\p^3,X_b)<0$ and for every $b\neq 0$
  $X_b$ is 
  Cremona equivalent to a plane, while $X_0$ is  not Cremona equivalent
  to a plane.
\end{theorem}
\begin{proof} To prove the theorem, we produce a linear system $\h$ on
  $V:=\p^2\times\Delta$, such that, with the notation of
  Lemma~\ref{lem:CEdeg6} and Theorem~\ref{thm:main}, for $t\neq 0$
  $\varphi_{\h_t}(\p^2)\cong S'$, while $\varphi_{\h_0}(\p^2)\cong
  S$.

  We already observed that a general sublinear
system of dimension $3$, $\Lambda_0\subset |\I_{p_1^3\cup p_2}(4)|$ is
such that
$$S=\varphi_{\Lambda_0}(\p^2)\subset\p^3.$$ 

First, we introduce a fixed component in the linear systems $\Lambda_0$.
Let $L\subset \p^2$ be a general line and $y_1,\ldots, y_n\in L$ general
points. Note that the linear system
$$\h_0:=\Lambda_0+L\subset|\o_{\p^2}(5)|$$
is such that
$\varphi_{\h_0}(\p^2)=\varphi_{\Lambda_0}(\p^2)=S$,  and $\h_0=|\I_{p_1^3\cup
  p_2\cup y_1\cup\ldots\cup y_h}(5)|$ for any $h\geq 6$.

We aim to obtain $\h_0$ as a specialization of linear systems in
$\p^2$ providing birational embedding of the surfaces described in Lemma~\ref{lem:CEdeg6}.

For this, keeping in  mind the Construction~\ref{not:collapse}, consider a degeneration $\pi:V\to\Delta$ together with the following sections:
\begin{itemize}
\item[-] $\{\sigma_{1},\ldots,\sigma_6\}$ such that $\sigma_i(0)=p_1$,
\item[-] $\{s_1, s_2, s_3\}$ with $s_i(0)\in L$,
\item[-] $\{y_1,\ldots, y_6\}$ with $y_i(t)\in L$, for any $t\in\Delta$,
  \item[-] $\{p_2\}$ with $p_2(t)=p_2$, for any $t\in\Delta$.
  \end{itemize}
  Let $Z=\cup\sigma_i\cup s_j \cup y_k\cup p_2$ and
  $\L=|\I_Z(5)|\subset|\o_V(5)|$ be the linear system containing $Z$.
Let
  $$\h\subset\L$$ be a general sublinear system of relative dimension
  $3$ then
  $$\h_t=\Lambda_t+L.$$
  For $t\neq 0$ the linear system  $\Lambda_t$ is contained in
  $|\I_{x_1\cup\ldots\cup x_{10}}(4)|$, with $x_i$ general points in
  $\p^2$ and, with the notation of Lemma~\ref{lem:CEdeg6},
  $\varphi_{\h_t}(\p^2)= S'$.
  For $t=0$, keeping in mind Construction~\ref{not:collapse}, we have
  $$\h_0=\Lambda_0+L\subset|\I_{p_1^3\cup p_2}(4)|+L.$$
  To conclude the proof, we need to show that, in the notation of
  Theorem~\ref{thm:main}, the special fiber is isomorphic
  to $S$. In other words  we have to
  prove that $\Lambda_0$ is a general linear space in $|\I_{p_1^3\cup p_2}(4)|$.
  Since $\dim|\I_{p_1^3\cup
    p_2}(4)|=7$, we have to produce a dominant map from the space
  parameterizing our degenerations to  ${\mathbb G}(3,7)$.
  
In the notation of Construction~\ref{not:collapse},
the restriction $\h_{Y|E}$ is given by cubics passing through $6$
general points. The linear system  $\h_{Y|E}$ is complete by  \cite[Lemma
24]{GM19}  and it is therefore the
dimension $3$  linear system of cubics
through the six points $\{\sigma_{Y,1}(0),\ldots,
\sigma_{Y,6}(0)\}$. In particular different  $6$-tuples of  points $\{\sigma_{Y,1}(0),\ldots,
\sigma_{Y,6}(0)\}$ produce different degenerations.
The linear system $\L_t$ has dimension $4$, therefore
its flat limit $\L_0$ has dimension $4$ and to any point 
$q\in\L_0^*\cong\p^4$ we may associate a sublinear system $\h_q\subset\L$
such that $[\h_{q0}]=q$. Again, changing the point produces different
limit linear systems $\h_0$.
Hence  the choice of the points $(\sigma_{Y,1}(0),\ldots,
\sigma_{Y,6}(0))\in(\p^2)^6$ and the point $q\in\p^4$  yields a generically finite map
$$\tau:(\p^2)^6\times\p^4\rat {\mathbb G}(3,7),$$
mapping the points to the limit linear system $\h_0$.
Both varieties have dimension $16$, then $\tau$
is dominant and,  for a general degeneration $\h$, we have $\varphi_{\h_0}(\p^2)=S$,
concluding the proof.
\end{proof}
\begin{remark}
  It is interesting to stress that it is not possible to provide the
  same degeneration with linear systems of relative degree $4$, indeed, there is
  no flat limit of $10$ general points to the scheme of length $7$
  given by a triple point
  and a simple point.  The trick is to force a  fixed component that absorbs the excess
  base points of the general fiber without changing the birational map.
\end{remark}

\bibliographystyle{amsalpha}
\bibliography{Biblio}

\end{document}